\documentclass[reqno,12pt]{amsart}
\usepackage[left=1.2in, right = 1.2in, top=1in]{geometry}

\usepackage[colorlinks=true]{hyperref} 

\hypersetup{urlcolor=blue, citecolor=red, linkcolor=blue}

\usepackage{amsmath, amssymb, amsthm, mathrsfs, color, times, textcomp, verbatim}
\usepackage{hyperref}
\usepackage{amstext}
\usepackage{appendix}
\usepackage{verbatim}
\usepackage{cases}
\usepackage{pgfplots}
\usepackage{graphicx}
\usepackage{tikz}
\newtheorem{theorem}{Theorem}[section]
\newtheorem{remark}{Remark}[section]

\newtheorem{lemma}{Lemma}[section]

\usepackage{fancyhdr}

\usepackage{mathtools}
\usepackage[shortlabels]{enumitem}

\usepackage[square,numbers]{natbib}

\usepackage{graphicx} % Required for inserting images

\title{The Neumann problem of  special Lagrangian type equations}
\author{Guohuan Qiu}
\author{Dekai Zhang}

\address{Institute of Mathematics, Academy of Mathematics and Systems Science, Chinese Academy of Sciences, No.55 Zhongguancun East Road, 100190, Beijing, China}
\email{qiugh@amss.ac.cn}

\address{School of Mathematical Sciences,  East China Normal Universit, Shanghai, China, 200241
}
\email{dkzhang@math.ecnu.edu.cn}

\begin{document}

\begin{abstract}
We study the Neumann problem for special Lagrangian type equations with critical and supercritical phases. These equations naturally generalize the special Lagrangian equation and the k-Hessian equation. By establishing uniform a priori estimates up to the second order, we obtain the existence result using the continuity method. The new technical aspect is our direct proof of 
boundary double normal derivative estimates. In particular, we directly prove the double normal estimates for the 2-Hessian equation in dimension 3. Moreover, we solve the classical Neumann problem by proving the uniform gradient estimate.\\
%\\
%\textbf{Keywords:} Special Lagrangian type equations; Neumann problems; Gradient estimates; Second order estimates
%\\
%\textbf{Mathematics Subject Classification} \quad	35J25 $\cdot$  35J60
\end{abstract}

\maketitle

\section{Introduction}
The special Lagrangian equation 
\begin{equation}
 \sum_{i=1}^{n}\arctan\lambda_{i} (D^2 u)=\Theta  \label{SLE}
\end{equation}
was introduced by Harvey-Lawson \cite{harvey1982calibrated} back in 1982. Its solution $u$ was demonstrated to possess the property that the graph $(x, \nabla u) \in \mathbb{R}^n \times \mathbb{R}^n$ forms a Lagrangian submanifold that is absolutely volume-minimizing.
The Dirichlet problem of this equation was solved by Caffarelli-Nirenberg-Spruck \cite{CNS-1,CNS-3} for $\Theta = \frac{(n-2)\pi}{2}$ when $n$ is even, and $\Theta= \frac{(n-1)\pi}{2}
$ when $n$ is odd, under a condition on the geometry of the domain $\Omega$. The existence and uniqueness of the Dirichlet problem for the viscosity solution of \eqref{SLE} were demonstrated by Harvey-Lawson in \cite{HarveyLawson09CPAM}, and smooth solutions for critical and supercritical phases were obtained by Yuan in \cite{Yuan16Slag}. The interior regularity of the special Lagrangian equations \eqref{SLE} for both critical and supercritical phases were proved Warren-Yuan \cite{warren2009hessian,WarrenYuan2010AJM} and Wang-Yuan \cite{WY11}. Chen-Warren-Yuan also obtained results for the convex case in \cite{chen2009priori} and \cite{ChenShankarYuan}. 
In \cite{BrendleWarrenJDG10}, Brendle-Warren studied a second boundary value problem for the special Lagrangian equation.
The special Lagrangian equation on a compact Kähler manifold also arises from mirror symmetry and is called the deformed Hermitian Yang-Mills equation, which was firstly studied by Jacob-Yau \cite{JacobYau2017MathAnn}. See \cite{CJY2020Cambridge, CollinsYau2021AnnPDE, ChenGao2021,Pingali2022APDE, FuYauZhang2021, Lin2023Adv} for recent progress.

One direction of generalization of the special Lagrangian equation is studied by considering  the Lagrangian mean curvature equation
\begin{equation}
    \sum\limits^n_{i=1} \arctan \lambda_i(D^2 u) = \Theta(x),\label{LMCE}
\end{equation}
where $(x,Du) \in \mathbb{R}^n \times \mathbb{R}^n $ is a submanifold with bounded mean curvature. The interior estimates and regularity of equation \eqref{LMCE} have been investigated in several works, including \cite{BhattacharyaCVPDE21, bhattacharya2022note, bhattacharya2022gradient, BhattacharyaShankarCrelle, bhattacharya2020optimal, zhou2023hessian}. The Dirichlet problem for equation \eqref{LMCE} has been addressed in \cite{CollinsPicardWu, DinewAPDE19, cirantpayne2021matheng,harveylawson2021cvpde, bhattacharya2022gradient}. Additionally, the second boundary problem of equation \eqref{LMCE} is proven in \cite{WangHuangBaoCVPDE}. 

Another similar generalization of the special Lagrangian equation is given by the following special Lagrangian-type equations:
\begin{align}
    \sum_i \arctan \frac{\lambda_i (D^2 u)}{f(x)} = \Theta. \label{slte1}
\end{align} 
Our motivation for studying the equation in the form \eqref{slte1} comes from an observation made by the first author while investigating the interior $C^{1,1}$ estimate of $\sigma_2 = f^2$. In \cite{qiu2017interior, qiu2019interior}, the graph $(x, Du)$, where $u$ satisfies the equation (\ref{slte1}), can also be regarded as a submanifold in $(\mathbb{R}^{n}\times\mathbb{R}^{n}, f^2 (x)dx^{2} + dy^{2})$ with bounded mean curvature.  The interior regularity of equation \eqref{slte1} was studied in \cite{qiu2017interior,zhou2023notes,lu2023interior}. Moreover, the algebraic form of equation \eqref{slte1} is
\begin{equation*}
   \cos\Theta \sum\limits _{1\leq 2k+1 \leq n}\frac{(-1)^{k}}{f(x)^{2k+1}} \sigma_{2k+1}(D^2 u)-\sin \Theta \sum\limits _{0 \leq 2k \leq n} \frac{(-1)^{k}}{f(x)^{2k}}  \sigma_{2k}(D^2 u)=0. 
\end{equation*}
So this is a special case of the mixed Hessian equations, as investigated by Krylov \cite{Krylov} and later by Guan-Zhang \cite{GuanZhang} in the following form:
\begin{equation}
\sigma_k (D^2 u) +\alpha(x) \sigma_{k-1} (D^2 u) = \sum^{k-2}_{l=0} \alpha_l (x) \sigma_k (D^2 u). \label{GuanZhang}
\end{equation}
These equations arise in various contexts, such as the problem of prescribing a convex combination of area measures. For more motivations behind studying equation \eqref{GuanZhang}, we refer to the paper by Guan-Zhang \cite{GuanZhang}.

In this paper, our aim is to explore the Neumann problem for the special Lagrangian type equation \eqref{slte1}. We consider it as a generalization of both the special Lagrangian equation and the $k$-Hessian equation, formulated as follows:
\begin{align}\label{GSLE}
\left\{
\begin{aligned}
    F(D^2u, x):=\sum_i \arctan \frac{\lambda_i (D^2 u)}{f(x)} = \Theta \quad in \quad \Omega, \\
    u_\nu = \varphi(x,u)\quad on \quad \partial \Omega,
    \end{aligned}
    \right.
\end{align}
where $\nu$ is the unit outward normal of $\partial \Omega$.

When $f=1$, this corresponds to the Neumann problem for the special Lagrangian equation. In the case of $n=2$ and $\Theta = \frac{\pi}{2}$, equation \eqref{SLE} transforms into the Monge-Amp\`ere equation. The Neumann problem for the Monge-Amp\`ere equation was successfully tackled by Lions-Trudinger-Urbas \cite{LTU1986}. When $n=3$ and $\Theta=\frac{\pi}{2}$, equation \eqref{slte1} is equivalent to $\sigma_2(D^2 u)=f^2$. In the paper \cite{TrudingerConj}, Trudinger raised the question of the solvability of the Neumann problem for the $k$-Hessian equation, spanning from balls to sufficiently smooth uniformly convex domains. This conjecture was later solved by Ma and the first author \cite{MaQiu2019CMP}. For the specific cases of $n=3$ with $\Theta =\pi$, or $n=4$ with $\Theta =\pi$, equation \eqref{slte1} can be expressed as $\frac{\sigma_3}{\sigma_1}=f^2$. Chen-Zhang \cite{ChenZhang2021BMS} extended Ma-Qiu's results to the Hessian quotient equation $\frac{\sigma_k }{\sigma_l} =f$. Regarding the special Lagrangian equation, the supercritical case, i.e., $\Theta>\frac{(n-2)\pi}{2}$, was resolved by Chen-Ma-Wei \cite{ChenMaWei2019JDE}, while the critical case, i.e., $\Theta=\frac{(n-2)\pi}{2}$, was addressed by Wang \cite{Wang2019CMS}. The Neumann problem for other types of equations has also been studied in \cite{ChenMaWei2019ATA, ChenMaZhang2021ActaSinica, Wang2019IJM}.

The Neumann boundary is another important boundary condition, aside from the Dirichlet boundary condition. It serves both as a condition for the existence of the equation and finds applications in proving isoperimetric inequalities. For instance, Cabre utilized the Neumann problem for the Laplace equation in \cite{Cabre08} to offer a straightforward new proof of the classical isoperimetric inequality. Additionally, in \cite{Brendle21}, Brendle employed solutions to the Neumann problem to establish the isoperimetric inequality for minimal subsubmanifolds in Euclidean Space. For fully nonlinear PDEs, solutions to the Neumann problem can also be employed to offer a new proof of Aleksandrov-Fenchel inequalities, as demonstrated by the first author and Xia in \cite{QiuXia19}, where the existence of the Neumann problem of $k$-Hessian equation was previously proven in \cite{MaQiu2019CMP}. Here, we consider the existence of the Neumann problem for the special Lagrangian type equation \eqref{slte1} and prove the following theorems.\\

\begin{theorem}\label{sltethm1}
Assume $\Omega\subset \mathbb{R}^n$ is a  strictly convex smooth domain.
Let $f\in C^{\infty}(\overline \Omega)$ be a positive function and $\phi\in C^{\infty}(\partial\Omega)$. 
Assume the constant $\Theta\in[\frac{(n-2) \pi}{2}, \frac{n\pi}{2})$. Then there exists a unique smooth solution solving
\begin{align}\label{slte22}
\left\{
\begin{aligned}
    \sum_{i=1}^n \arctan \frac{\lambda_i (D^2 u )}{f(x)} = \Theta \quad in \quad \Omega, \\
    u_\nu = -u+ \phi(x)\quad on \quad \partial \Omega.
    \end{aligned}
    \right.
\end{align}
Moreover, we have the following $C^{1,1}$ estimate up to the boundary
\begin{equation*}
    \lVert u \rVert_{C^{1,1}(\overline{\Omega})} \leq C(\lVert f^{-1} \rVert_{L^\infty}, \lVert f \rVert_{C^{1,1}(\overline{\Omega})},n, \lVert \phi \rVert_{C^{3}(\overline{\Omega})} ).
\end{equation*}

\end{theorem}
\begin{remark}
    The new technical aspect is our direct proof of 
boundary double normal derivative estimates. In particular, we directly prove the double normal estimates for the 2-Hessian equation in dimension 3.
\end{remark}
Our proof of Theorem \ref{sltethm1} is primarily based on the work of Ma-Qiu \cite{MaQiu2019CMP}. In \cite{MaQiu2019CMP}, they employ Lions-Trudinger-Urbas's technique in \cite{LTU1986CPAM} to transform the second order estimates from the interior to the boundary double normal derivative estimate. Then they construct a barrier function to establish the boundary double normal derivative estimate through an involved argument. The novel technical aspect of the proof of Theorem \ref{sltethm1} lies in utilizing the special properties of the special Lagrangian equation to provide a simplified proof of the boundary double normal derivative estimate before establishing the global second-order estimate, as detailed in subsection \ref{DDN}. 

Next we solved the classical Neumann problem for the special Lagrangian type equation. 
\begin{theorem}\label{sltethm2}
Assume $\Omega\subset \mathbb{R}^n$ is a  strictly convex smooth domain.
Let $f\in C^{\infty}(\overline \Omega)$ be a positive function and $\phi\in C^{\infty}(\overline \Omega)$. 
Assume the constant $\Theta\in[\frac{(n-2)\pi}{2}, \frac{n\pi}{2})$. Then there exist a unique smooth solution 
$u$ up to a constant and a unique constant $\lambda$ solving the following problem
\begin{align}\label{slte11}
\left\{
\begin{aligned}
    \sum_{i=1}^n \arctan \frac{\lambda_i (D^2 u)}{f(x)} = \Theta \quad in \quad \Omega, \\
    u_\nu = \lambda+ \phi(x) \quad on \quad \partial \Omega.
    \end{aligned}
    \right.
\end{align}

\end{theorem}
To solve the above classical problem, we consider the following approximating equation.
\begin{align}\label{aslte11}
\left\{
\begin{aligned}
    \sum_{i=1}^n \arctan \frac{\lambda_i (D^2 u^{\varepsilon})}{f(x)} = \Theta \quad in \quad \Omega, \\
    u^\varepsilon_\nu =  -\varepsilon u^\varepsilon+\phi(x)\quad on \quad \partial \Omega.
    \end{aligned}
    \right.
\end{align}
We will prove $|Du^\varepsilon|$ and $|D^2u^\varepsilon|$ have uniform bounds which are independent of $\varepsilon$.

{\bf{Notations}} In this paper, $C$ is a uniformly positive constant depending only on $n, \Omega, |f|_{C^2}$, $|f^{-1}|_{L^\infty}, |\varphi|_{C^{2}}$. For two positive functions $g,h$, $g\sim h$ means there exists a positive uniform constant $C$ such that $C^{-1}h\le g\le Ch$.

\section{Preliminaries}
\subsection{Equations from differentiating the special Lagrangian type equation}

\begin{lemma}
Let $W=\{ W_{ij} \}$ is a $n\times n$ symmetric matrix and $\lambda(W)=(\lambda_1,\lambda_2, \cdots, \lambda_n)$ are eigenvalues of the symmetric matrix $W$. Suppose that $W$ is diagonal and $\lambda_i = W_{ii}$, then we have
\begin{align*}
    &\frac{\partial \lambda_i}{\partial W_{ii}} =1 , \quad \frac{\partial \lambda_k}{\partial W_{ij}} =0 \quad otherwise,\\
    &\frac{\partial^2 \lambda_i}{\partial W_{ij} \partial W_{ji}} =\frac{1}{\lambda_i -\lambda_j}, \quad i\neq j \quad and \quad \lambda_i \neq \lambda_j\\
    &\frac{\partial^2 \lambda_i}{\partial W_{kl} \partial W_{pq}}=0 \quad otherwise.
\end{align*}
\end{lemma}
Differentiating the equation \eqref{slte1}, 
\begin{align*}
&\sum_{i,j=1}^nF^{ij}u_{ijp}+F_{x_p}=0,\\
&\sum_{i,j=1}^nF^{ij}u_{ijpq}+\sum_{i,j,k,l=1}^nF^{ij,kl}u_{ijp}u_{klq}+\sum_{i,j=1}^nF^{ij}_{x_q}u_{ijp}+\sum_{i,j=1}^nF^{ij}_{x_p}u_{ijq}+F_{x_p,x_q}=0.
\end{align*}

If $D^2 u(x_0)=\{\lambda_i\delta_{ij}\}$ is diagonal, we have
\begin{align*}
F^{ij}=& \frac{f}{f^2+\lambda_i^2}\delta_{ij},\\
F^{ij}_{x_p}=&-\frac{f_p\delta_{ij}}{f^2+\lambda_i^2}+\frac{2\lambda_i^2f_p\delta_{ij}}{(f^2+\lambda_i^2)^2},\\
F_{x_p}=&-\sum_{i=1}^n\frac{\lambda_i}{f^2+\lambda_i^2}f_p,\\
F_{x_p,x_q}=&\sum_{i=1}^n\frac{\lambda_i}{f^2+\lambda_i^2}\Big(-\frac{2\lambda_i^2f_pf_q}{f(f^2+\lambda_i^2)}+\frac{2f_pf_q}{f}-f_{pq}\Big),
\end{align*}

\begin{align*}
F^{ij,kl}=& \sum_p -\frac{2f\lambda_p}{(f^2 + \lambda^2_p)^2} \frac{\partial \lambda_p}{\partial u_{kl}} \frac{\partial \lambda_p}{\partial u_{ij}}+ \frac{f}{f^2+\lambda^2_p} \frac{\partial^2 \lambda_p}{\partial u_{ij} \partial u_{kl}} .
\end{align*}
If $i=j=k=l$, 
\begin{equation*}
    F^{ij,kl} = -\frac{2f\lambda_i}{(f^2+\lambda^2_i)^2}.
\end{equation*}
If $i=l, k=j$ and $i\neq j$, 
\begin{align*}
    F^{ij,kl} = &\frac{f}{(f^2+\lambda^2_i) (\lambda_i -\lambda_j)} + \frac{f}{(f^2+\lambda^2_j) (\lambda_j -\lambda_i)}\\
    =& \frac{f(\lambda^2_j-\lambda^2_i)}{(f^2+\lambda^2_i)(f^2+\lambda^2_j)(\lambda_i -\lambda_j)}\\
    =& -\frac{f(\lambda_i+\lambda_j)}{(f^2+\lambda^2_i)(f^2+\lambda^2_j)}.
\end{align*}
Thus we have
\begin{align*}
F^{ij,kl}=\left\{
\begin{aligned}
&-\frac{f(\lambda_i+\lambda_j)}{(f^2+\lambda_i^2)(f^2+\lambda_j^2)},& i=l, k=j,\\
&0,\qquad & \text{otherwise}.
\end{aligned}
\right.
\end{align*}
Then we have
\begin{align*}%\label{1stderi}
\sum_{i,j=1}^nF^{ij}u_{ijp}=\sum_{i=1}^n\frac{f}{f^2+\lambda_i^2}u_{iip}=\sum_{i=1}^n\frac{\lambda_i}{f^2+\lambda_i^2}f_p,
\end{align*}
and
\begin{align*}%\label{2edderi}
\sum_{i,j=1}^nF^{ij}u_{ijpp}=&\frac{1}{f}\sum_{i=1}^n\frac{2\lambda_i}{(f^2+\lambda_i^2)^2}(fu_{iip}-\lambda_if_p)^2+\sum_{i\neq j}\frac{f(\lambda_i+\lambda_j)}{(f^2+\lambda_i^2)(f^2+\lambda_j^2)}u_{ijp}^2
\notag\\
&+2f_p\sum_{i=1}^n\frac{u_{iip}}{f^2+\lambda_i^2}+\sum_{i=1}^n\frac{\lambda_i}{f^2+\lambda_i^2}\Big(f_{pp}-\frac{2f_p^2}{f}\Big).
\end{align*}

In conclusion, we have
\begin{align}
\sum_{i,j=1}^nF^{ij}u_{ijp}=&\sum_{i=1}^n\frac{\lambda_i}{f^2+\lambda_i^2}f_p \quad \text{i.e.} \quad \sum_{i=1}^n\frac{fu_{iip}-\lambda_i f_p}{f^2+\lambda_i^2}=0\label{diff1time} \\
\sum_{i,j=1}^n F^{ij}u_{ijpp}=&\frac{1}{f}\sum_{i=1}^n\frac{2\lambda_i}{(f^2+\lambda_i^2)^2}(fu_{iip}-\lambda_if_p)^2+\sum_{i\neq j}\frac{f(\lambda_i+\lambda_j)}{(f^2+\lambda_i^2)(f^2+\lambda_j^2)}u_{ijp}^2
\notag\\
&+\sum_{i=1}^n\frac{\lambda_i}{f^2+\lambda_i^2}f_{pp}\label{diff2times}.
\end{align}
For special Lagrangian type equations, these properties are well-known and can be found in \cite{WarrenYuan2010AJM,WY11}.
\begin{lemma} \label{WWY}
Suppose that $+\infty>\lambda_1 \geq \lambda_2 \geq \cdots \geq \lambda_n$ satisfy $\sum\limits_{i} \arctan \frac{\lambda_i}{f}=\Theta \geq (n-2)\frac{\pi}{2}$ and $f>0$. The following properties hold
\begin{enumerate}[(1)]
    \item $\lambda_1 \geq \lambda_2 \geq \cdots \geq\lambda_{n-1} >0$, $|\lambda_n|\leq \lambda_{n-1}$.
    \item \label{1lambda}if $\lambda_n < 0$, then $\sum^n_{i=1} \frac{1}{\lambda_i}\leq 0$.
    \item \label{deltalambda} If $\sum\limits_{i} \arctan \frac{\lambda_i}{f} \geq (n-2)\frac{\pi}{2} + \delta$, then $\lambda_n \geq -C(\delta)\max |f|$.
\end{enumerate}
 
\end{lemma}

\begin{proof}
Let us denote 
\begin{equation*}
    \theta_i =\arctan \frac{\lambda_i}{f}.
\end{equation*}
Thus our equation \eqref{slte1} is 
\begin{equation*}
    \sum\limits_i \theta_i =\Theta \geq \frac{(n-2)\pi}{2}.
\end{equation*}
We assume that $\theta_1 \geq \theta_2 \geq \cdots \theta_{n-1} \geq \theta_n$. Thus it is not hard to see that 
\begin{equation*}
    \theta_{n-1} + \theta_n \geq 0.
\end{equation*}
Otherwise, we would have $ \sum\limits_i \theta_i < \frac{(n-2)\pi}{2}$, which contradicts our equation \eqref{slte1}.
Thus there are at least $n-1$ finite positive eigenvalues, say $\lambda_1 \geq \lambda_2 \geq \cdots \geq\lambda_{n-1} >0 $ and $|\lambda_n|\leq \lambda_{n-1}$. 
When $\Theta\geq \frac{(n-2)\pi}{2}$ and $\lambda_n<0$, we have 
\begin{equation*}
    \frac{\pi}{2} > \frac{\pi}{2} +\theta_n \geq  \sum\limits^{n-1}_{i=1}(\frac{\pi}{2} -\theta_i)>0.
\end{equation*}
By an elementary identity for $\tan$ function, we have 
\begin{eqnarray*}
    \tan \sum\limits^{n-1}_{i=1}(\frac{\pi}{2} -\theta_i) &=&\frac{\tan (\frac{\pi}{2}-\theta_1) + \tan \sum\limits^{n-1}_{i=2}(\frac{\pi}{2} -\theta_i )}{1-\tan (\frac{\pi}{2}-\theta_1) \tan \sum\limits^{n-1}_{i=2}(\frac{\pi}{2} -\theta_i )}\\
    &\geq& \tan (\frac{\pi}{2}-\theta_1) + \tan \sum^{n-1}_{i=2}(\frac{\pi}{2} -\theta_i )\\
    & \geq& \vdots\\
    &\geq& \sum\limits^{n-1}_{i=1} \tan(\frac{\pi}{2}-\theta_i).
\end{eqnarray*}
    
Thus 
\begin{equation*}
    -\frac{f}{\lambda_n}=\tan (\frac{\pi}{2} + \theta_n)  \geq \tan  \sum\limits^{n-1}_{i=1}(\frac{\pi}{2} -\theta_i)\geq \sum\limits^{n-1}_{i=1} \tan(\frac{\pi}{2}-\theta_i) = \sum\limits^{n-1}_{i=1} \frac{f}{\lambda_i}.
\end{equation*}
So we have 
\begin{equation*}
    \sum^n_{i=1}\frac{f}{\lambda_i}\le 0.
\end{equation*}   
As $f > 0$, the inequality above is equivalent to
\begin{equation*}
\sum^n_{i=1}\frac{1}{\lambda_i} \le 0.
\end{equation*}
Suppose  $\sum\limits_{i} \arctan \frac{\lambda_i}{f} \geq (n-2)\frac{\pi}{2} + \delta$, we should have 
\begin{equation*}
    \theta_n \geq -\frac{\pi}{2} + \delta.
\end{equation*}
Thus we obtain
\begin{equation*}
    \lambda_n \geq-C(\delta)\max |f|.
\end{equation*}
 
\end{proof}

The following lemma will be used to derive the global second order derivative estimate.
\begin{lemma}\label{deri12}
Let $u$ be a solution of the special Lagrangian type equation \eqref{slte1}. Assume $\{D^2u(x_0)\}=\{\lambda_i\delta_{ij}\}$ with $\lambda_1\ge \lambda_2\ge \cdots \geq \lambda_n$, then we have
\begin{align}
&\Big|\sum_{i,j=1}^n F^{ij}u_{ijp}\Big|\le |Df|\big| \sum_{i=1}^n\frac{\lambda_i}{f^2+\lambda_i^2}\big|,\label{deri1}\\
&\sum_{i,j=1}^nF^{ij}u_{ijpp}\ge -|D^2f|\big| \sum_{i=1}^n\frac{\lambda_i}{f^2+\lambda_i^2}\big|.\label{deri2}
\end{align}
\end{lemma}
\begin{proof}
The inequality \eqref{deri1} follows from \eqref{diff1time}.

To prove \eqref{deri2}, by \eqref{diff2times} and $\lambda_i+\lambda_j\ge 0, \ \forall\  i\neq j$ , we get
\begin{eqnarray*}
\sum_{i,j=1}^nF^{ij}u_{ijpp}\ge& \sum_{i=1}^n\frac{\lambda_i}{(f^2+\lambda_i^2)^2}(fu_{iip}
-\lambda_if_p)^2-|D^2f|\big|\sum\limits_{i=1}^n\frac{\lambda_i}{f^2+\lambda_i^2}\big|.
\end{eqnarray*}
Then we only need to prove $$\mathcal{I}:=\sum_{i=1}^n\frac{\lambda_i}{(f^2+\lambda_i^2)^2}(fu_{iip}-\lambda_if_p)^2\ge 0.$$
This is obvious if $\lambda_n\ge 0$. 
Now we assume $\lambda_n<0$.  Set  $x_i=\frac{fu_{iip}-\lambda_i f_p}{f^2+\lambda_i^2}$, then  $x_n=-\sum_{i=1}^{n-1}x_i$. Since $\lambda_i>0, 1\le \lambda_i\le n-1$, by Cauchy inequality, we get
\begin{align*}
\mathcal{I}=\ &\sum_{i=1}^{n-1}\lambda_i x_i^2+\lambda_n(\sum_{i=1}^{n-1}x_i)^2\\
\ge\ & \frac{(\sum\limits_{i=1}^{n-1}x_i)^2}{\sum\limits_{i=1}^{n-1}{\lambda_i^{-1}}}+\lambda_n(\sum_{i=1}^{n-1}x_i)^2\\
=\ &\frac{(\sum\limits_{i=1}^{n-1}x_i)^2}{\sum\limits_{i=1}^{n-1}{\lambda_i^{-1}}}\lambda_n\sum\limits_{i=1}^{n}\lambda_i^{-1}\\
\ge\ & 0,
\end{align*}
where the last inequality is a consequence of \ref{1lambda} from Lemma \ref{WWY}.

\end{proof}

For the boundary estimates in section 3 and 4, we define
\begin{equation*}
\Omega_{\mu}={x\in\Omega: d(x,\partial\Omega)<\mu},
\end{equation*}
and let
\begin{equation}\label{}
h(x) = -d(x) + d^2(x).
\end{equation}
It is known from the classic book \cite{GT} section 14.6 that $h$ is $C^4$ in $\Omega_{\mu}$
for some constant $\mu \leq \widetilde{\mu}$, where $\widetilde{\mu}$ depends on $\Omega$.
In terms of a principal coordinate system, see \cite{GT}  section 14.6, for any $x_0\in\Omega_{\mu}$, there exists a unique point $y_0\in\partial\Omega$ such that 
\begin{equation*}
  \{-D^2d(x_0)\} =\mathrm{diag}\{\frac{\kappa_1(y_0)}{1-\kappa_1(y_0) d(x_0)}, \cdots , \frac{\kappa_{n-1}(y_0)}{1- \kappa_{n-1}(y_0)d(x_0)}, 0\},
\end{equation*}
and
\begin{equation}\label{}
  -Dd(x_0) = \nu(y_0) = (0,0,\cdots,1),
\end{equation}
where  $\nu$ is the unit outward normal on the boundary $\partial \Omega$.
Then $h$  satisfies the following properties in $\Omega_{\mu}$:
\begin{align*}
   -\mu + \mu^2 \leq h \leq& 0,\\
  \frac{1}{2} \leq |D h| \leq 2,\\
  \kappa_0I \leq D^2 h 
  \leq& K_0 I,\\
\sum\limits_{i,j}F^{ij} h_{ij} \geq& k_0 \mathcal{F},
\end{align*}
provided $\mu \leq \widetilde{\mu}$ small depend on $||\partial \Omega||_{C^2}$. Here $\kappa_0$ and $K_0$ are positive constants depending on $\kappa : = (\kappa_1, \cdots, \kappa_{n-1})$.
It is easy to see
\begin{equation}\label{}
  {Dh} = \nu\quad\text{on} \quad\partial\Omega.
\end{equation}
.

\section{$C^0$ and $C^1$-estimates}

In this section, we prove the $C^0$ estimate and gradient estimate for the Neumann problem \eqref{GSLE}.
 The gradient estimate contains 
 interior gradient estimates  and the  near boundary gradient estimates for equation \eqref{GSLE}.
We also prove the uniform gradient estimate for the classical problem by assuming the strict convexity of the domain.

\subsection{$C^0$-estimate}

\begin{theorem}
$(1)$
Let 
$u$ be a $C^2$ solution of problem \eqref{GSLE} with $-\varphi_u\ge c_0>0$.
We have
\begin{equation}
\max_{\overline \Omega} |u|\le C.\label{C^0}
\end{equation}
$(2)$ 
Let $u$ be a $C^2$ solution of problem \eqref{aslte11}. We have
\begin{equation}
\max_{\overline \Omega} |\varepsilon u^\varepsilon|\le C. \label{epsilonU} 
\end{equation}
\end{theorem}
\begin{proof}
 Since $u$ is subharmonic,  $u$ attains its maximum at $x_0\in \partial \Omega$. We assume $u(x_0)>0$ otherwise $u$ has an upper bound. Then we have
\begin{align*}
0\le u_{\nu}=&\varphi(x_0,u(x_0))-\varphi(x_0,0)+\varphi(x_0,0)\\
=&u\varphi_u(x_0,tu(x_0))+\varphi(x_0,0)\\
\le& -c_0 u+\varphi(x_0,0).
\end{align*}
This gives the uniform upper bound of $u$.

Next, we prove the lower bound. 
Let $v=C_0\frac{|x|^2}{2}$ with $C_0=|f|_{C^0}(\tan\frac{\Theta}{n}+1)$. We have
\begin{align*}
F(D^2v,x)=n\arctan\frac{C_0}{f}>\Theta=F(D^2u, x).
\end{align*}
By the maximum principle, $u-v$ attains its minimum on $x_1\in \partial \Omega$. We assume $u(x_1)<0$ otherwise $\min u\ge -v(x_1)+\max v$. Then we have 
\begin{align*}
0\ge u_{\nu}(x_1)-v_{\nu}(x_1)=&
\varphi(x_1,u(x_1))-\varphi(x_1,0)+\varphi(x_1,0)-v_{\nu}(x_1)\\
=&\varphi_u(x_1,t_1u(x_1))u(x_1)+\varphi(x_1,0)-v_{\nu}(x_1)\\
\ge& -c_0 u(x_1)+\varphi(x_1,0)-v_{\nu}(x_1).
\end{align*}
Then we have $u(x_1)\geq -C$ and $u(x)\ge v(x)+ u(x_1)-v(x_1)\ge -C$.
The proof of \eqref{epsilonU} is similar to that of \eqref{C^0}, so we omit it.
\end{proof}

\subsection{The gradient estimate}
In this subsection, We will prove the interior gradient estimate and the boundary gradient estimate. One can see the gradient estimates for k-Hessian curvature equations with prescribed contact angle by Deng-Ma \cite{DM2023ME} and k-Hessians equation with oblique boundary condition by Wang \cite{Wang2022ANS}.

When $\Theta\in[(n-2)\frac{\pi}{2},n\frac{\pi}{2})$ is a constant and $f=1$, the interior gradient estimate was proved by Warren-Yuan \cite{WarrenYuan2010AJM}. When $f=1$ and $\Theta(x)\in[(n-2)\frac{\pi}{2},n\frac{\pi}{2})$, it was proved by Bhattacharya-Mooney-Shanker\cite{bhattacharya2022gradient}. For the special Lagrangian type equation, we will show the following interior gradient estimate hold.
\begin{theorem}
Let $u$ be a solution of the special Lagrangian type equation \eqref{slte1} in $B_1(0)$. Then there exists a positive constant $C$ such that
\begin{equation}
\sup_{B_{\frac{1}{2}}(0)} |Du|\le C(\sup_{B_1}u-\inf_{B_1}u+1)\log(\sup_{B_1}u-\inf_{B_1}u+1). \label{Interiorgradient}
\end{equation}
\end{theorem}
\begin{proof}
We consider
\begin{equation*}
G(x,\xi)=|Du|\eta+g(u),
\end{equation*}
where $\eta=\frac{1-|x|^2}{2}$ .
Assume $G$ attains its maximum at $x_0\in B_{1}$.  By rotating the coordinate, we assume $D^2u(x_0)$ is diagonal and denote $u_{ii}(x_0)$ by $\lambda_i$.\\
In the following, all the calculations are at $x_0$.
Firstly, we have 
\begin{align}\label{ig129}
0=G_i=&|Du|_{i}\eta+|Du|\eta_{i}+g'u_{i}\notag\\
=&\frac{u_{k}u_{ki}}{|Du|}\eta-|Du|x_{i}+g'u_{i}.
\end{align}
Without loss of generality, we assume $u_{n} \ge \frac{1}{\sqrt{n}}|Du| >0$. Then by \eqref{ig129} and choosing $g'>2n$, we get
\begin{align}\label{lamn129}
\frac{1}{2}g'u_{n}\le (-\lambda_n)\eta\le 2\sqrt{n} g'u_{n}.
\end{align}
Since $\Theta\ge (n-2)\frac{\pi}{2}$, we have  from Lemma \ref{WWY} that \begin{align*} %\label{lami129}
\lambda_i\ge |\lambda_n|\ge \frac{1}{2}g'u_{n}. \end{align*} 
Denote $\mathcal{F}:= \sum\limits^n_{i=1} F^{ii}$, we have
\begin{align}%\label{Fsum129}
\mathcal{F}\sim F^{nn}\sim|\lambda_n|^{-2}\sim u_{n}^{-2}\eta^2.
\end{align}
By the maximum principle, we have
\begin{align}\label{FG129}
0\ge \sum_{ij}F^{ij}G_{ij}=&\eta \sum_{i}F^{ii}|Du|_{ii}-2\sum_{i}F^{ii}\frac{\lambda_{i}u_{i}}{|Du|} x_i -|Du|\mathcal{F}\notag\\
&+g'\sum_{i}F^{ii}\lambda_i+g''\sum_{i}F^{ii}u_{i}^2
\end{align}
Firstly we estimate $\sum\limits_{i}F^{ii}|Du|_{ii}$ as follows:
\begin{align}\label{FiiDuii129}
\sum_{i}F^{ii}|Du|_{ii}=&\sum_{i,k}\frac{F^{ii}u_{iik}u_k}{|Du|}+\sum_{i}F^{ii}\frac{u_{ii}^2}{|Du|}-\sum_{i}F^{ii}\frac{u_{i}^2 u_{ii}^2}{|Du|^3}\notag\\
\ge& -\sum_{k}|\sum_{i}F^{ii}u_{iik}|\notag\\
\ge&-|D\log f|\sum_{i}F^{ii}|\lambda_i|,
\end{align}
where we use $\sum\limits_{i}F^{ii}u_{iik}=(\log f)_{k}\sum\limits_{i}F^{ii}\lambda_i$ in the last inequality.\\
The good term $\sum\limits_{i}F^{ii}u_i^2$ have the following estimate
\begin{align}\label{Fiiui2}
\sum_{i}F^{ii}u_i^2\ge F^{nn}u_{n}^2\ge \frac{1}{n^2}\mathcal{F}|Du|^2.
\end{align}
Inserting \eqref{FiiDuii129} and \eqref{Fiiui2} into \eqref{FG129}, we have
\begin{align*}
0\ge\sum_{ij}F^{ij}G_{ij}\ge& \frac{g''}{n^2}\mathcal{F}|Du|^2-|Du|\mathcal{F}-(g'+C)\sum_{i}F^{ii}|\lambda_{i}|\\
\ge&\frac{g''}{n^2}\mathcal{F}|Du|^2-|Du|\mathcal{F}-C(g'+1)\mathcal{F}^{\frac{1}{2}},
\end{align*}
where we use $\sum\limits_{i}F^{ii}|\lambda_i|\le (\sum\limits_{i}F^{ii}\lambda_i^2)^{\frac{1}{2}}(\sum\limits_{i}F^{ii})^{\frac{1}{2}}\le \sqrt{nf}\mathcal{F}^{\frac{1}{2}}$ in the last inequality.
\\
Then we get
\begin{align*}
\frac{g''}{n^2}|Du|^2-|Du|\le& C(g'+1)\mathcal{F}^{-\frac{1}{2}}\\
\le& C(g'+1)|\lambda_n|\\
\le& Cg'(g'+1)|Du|\eta^{-1},
\end{align*}
where in the last inequality we use \eqref{lamn129}.
This implies 
\begin{align}\label{Du129}
|Du|\eta\le C\frac{1+(g')^2}{g''}.
\end{align}
If we choose $g=-A_0M\log(\sup_{B_1} u+1- u)$ with $M=\sup_{B_1} u-\inf_{B_1} u+1$, by \eqref{Du129}, we obtain
\begin{align}
|Du|(x_0)\eta(x_0) \le CA_0 M.
\end{align}
Then for any $x_0\in B_{\frac{1}{2}}$, we have
\begin{align}
\frac{1}{4}|Du|(x)\le \eta(x)|Du|(x)\le \eta(x_0)|Du|(x_0)+A_0M\log M\le CM(1+\log M).
\end{align}
\end{proof}

Set $\Omega_{\mu}=\{x\in \Omega: d(x):=dist(x, \partial \Omega)<\mu\}$ with $\mu$ a small positive constant. By choosing $\mu$ small enough depending only on $|D\varphi|_{C^0}$, we have 
\begin{align}
1+\varphi_u d\in \Big(\frac{2}{3}, \frac{4}{3}\Big). 
\end{align}

\begin{lemma}
\label{FijlogDw}
Set $w=u+\varphi d$.
Assume $|Dw|\sim|Du|>2$, there exists uniform constant $C$ such that 
\begin{align}
\sum_{i,j}F^{ij}(\log|Dw|^2)_{ij}\ge& -C(d+|Du|^{-1}) 
\sum_{i}{F^{ii}}|\lambda_i|-C(d|Du|^2+|Du|)\mathcal{F}\notag\\
&-\frac{1}{2}\sum_{i,j}F^{ij}\frac{(|Dw|^2)_{i}(|Dw|^2)_{j}}{|Dw|^4}.
\end{align}
\end{lemma}
\begin{proof}
By direct calculation, we have
\begin{align}
F^{ij}(\log|Dw|^2)_{ij}=& \sum_{i,j,k}\frac{2w_{k}F^{ij}w_{ijk}}{|Dw|^2}+\sum_{i,j,k}\frac{2F^{ij}w_{ki}w_{kj}}{|Dw|^2}-\sum_{i,j}F^{ij}\frac{(|Dw|^2)_{i}(|Dw|^2)_{j}}{|Dw|^4}\notag\\
\ge& \sum_{i,j,k}\frac{2w_{k}F^{ij}w_{ijk}}{|Dw|^2}-\sum_{i,j}\frac{1}{2}F^{ij}\frac{(|Dw|^2)_{i}(|Dw|^2)_{j}}{|Dw|^4},\label{Fijlog}
\end{align}
where we use Cauchy inequality in the last inequality.
To prove the lemma, we only need to estimate the first term on the right side of the above inequality.\\
By direct calculations, we get
\begin{eqnarray*}
w_{i}&=&(1+\varphi_{u}d)u_{i}+\varphi_{i}d+\varphi d_{i}=(1+\varphi_{u}d)u_{i}+O(1),
\\
w_{ij}&=&(1+\varphi_{u}d)u_{ij}+(\varphi_{uj}d +\varphi_{uu}u_j d+\varphi_{u}d_{j})u_{i}\\
& &+ 
(\varphi_{iu}d+\varphi_{u} d_{i} )u_{j}+\varphi_{ij} d+ \varphi_{i}d_{j}+ \varphi_{j} d_i+\varphi d_{ij} ,\\
w_{ijk}&=& (1+\varphi_u d ) u_{ijk} + (\varphi_{uu}  u_k  d+ \varphi_{uk} d + \varphi_u d_k) u_{ij}+(\varphi_{uu}  u_j d + \varphi_{uj} d + \varphi_u d_j) u_{ik}\\
& & +(\varphi_{uu}  u_i d + \varphi_{ui} d + \varphi_u d_i) u_{jk}+ \varphi_{uuu} u_i u_j u_k d+ \varphi_{uuj}  u_k u_i d + \varphi_{uuk}  u_j u_i d + \varphi_{uui}  u_k u_j d\\
& & +\varphi_{uu} u_j u_i d_k + \varphi_{uu} u_k d_j u_i+ \varphi_{uu} u_k d_i u_j + (\varphi_{u jk} d + \varphi_{uj} d_k+\varphi_{uk}d_j + \varphi_u d_{jk} )u_i\\
& & +(\varphi_{uik}d+\varphi_{ui}d_k+\varphi_{uk}d_i+\varphi_u d_{ik})u_j+(\varphi_{uij}d+\varphi_{ui}d_j+\varphi_{uj}d_i+\varphi_u d_{ij})u_k \\
& & +\varphi_{ijk} d + \varphi_{ij}d_k + \varphi_{ik} d_j+ \varphi_{jk}d_i + \varphi_i d_{jk}+ \varphi_j d_{ik}+ \varphi_k d_{ij}+\varphi d_{ijk}\\
&=&(1+\varphi_{u}d)u_{ijk}+O(d|Du|+1)(|u_{ij}|
+|u_{ik}|+|u_{jk}|)+O(d|Du|^3)\\
&&+O(|Du|^2+|Du|+1).
\end{eqnarray*}
Then we have
\begin{align*}
\sum_{i,j,k}w_{k}F^{ij}w_{ijk}=& (1+\varphi_{u}d)\sum_{i,j,k}F^{ij}u_{ijk}w_{k}+O(d|Du|^2+|Du|){F^{ii}}|u_{ii}|\\
&+O(d|Du|^4+|Du|^3+|Du|^2+|Du|)\mathcal{F}\\
=&(1+\varphi_{u}d)\sum_{i}\frac{\lambda_i}{f^2+\lambda_i^2}\sum_{k}f_kw_{k}+O(d|Du|^2+|Du|){F^{ii}}|u_{ii}|\\
&+O(d|Du|^4+|Du|^3+|Du|^2+|Du|)\mathcal{F}
\\
=&O(d|Du|^2+|Du|)\sum_{i}{F^{ii}}|\lambda_i|+O(d|Du|^4+|Du|^3+|Du|^2+|Du|)\mathcal{F}.
\end{align*}
Thus we get 
\begin{align*}
\sum_{i,j,k}\frac{w_{k}F^{ij}w_{ijk}}{|Dw|^2}\ge -C(d+|Du|^{-1})
\sum_{i}{F^{ii}}|\lambda_i|-C(d|Du|^2+|Du|)\mathcal{F}.
\end{align*}
Combining the above inequality with \eqref{Fijlog}, we get the lemma.
\end{proof}

Next we prove the near boundary gradient estimate. 
We need a lemma due to Warren-Yuan \cite{WarrenYuan2010AJM}.
\begin{lemma}[\cite{WarrenYuan2010AJM}]
\label{WY2010}
If $\Theta=(n-2)\frac{\pi}{2}$, we have
$\sum\limits_{i}F^{ii}\lambda_i\ge 0$.
\end{lemma}
We use the auxiliary function from Ma-Xu in \cite{MaXuAdv16} to prove the following near boundary gradient estimate and thus we get the global gradient estimates.
\begin{theorem}
Let 
$u$ be a $C^3$ solution of problem \eqref{GSLE}. There exists a uniform constant $C$ depending on $|\varphi|_{C^3},|f|_{C^1},|f^{-1}|_{L^\infty}$, $n,|\partial \Omega|_{C^2},|u|_{C^0}$ such that 
\begin{align}
\max_{ \overline\Omega_{\mu}} |Du|\le C.
\end{align}
\end{theorem}
\begin{proof}
We consider the auxiliary function in $\overline\Omega_{\mu}$
\begin{align}
    G=\log|Dw|^2-\log(M_0-u)+a_0 d,
\end{align}
where $w=u+\varphi d(x)$ and $M_0=|u|_{C^0}+1$.
Assume $G(x_0)=\max_{x\in \overline \Omega_{\mu}} G(x)$. We divide the following three cases to derive the  estimate:

{\bf{Case 1:} $x_0\in \partial\Omega_{\mu}\cap \Omega:=\{x\in\Omega: d(x)=\mu\}$}\\
This follows from the interior gradient estimate \eqref{Interiorgradient}.

{\bf{Case 2:} $x_0\in \partial\Omega$}\\
By choosing $\alpha_0$ large, the estimate follows from $G_{\nu}(x_0)\ge 0$ which is the same as in \cite{MaXuAdv16}.

{\bf{Case 3:} $x_0\in \Omega_{\mu}$}

{{The key point of the proof is the following: \\
We choose the coordinate such that $\{D^2 u(x_0)\}$ is diagonal.
W.L.O.G. we may assume $u_n\ge \frac{1}{n}|Du|$. Then $w_n\sim u_n\sim |Du|$. $G_n=0$ implies $u_{nn}<0$.
Thus $F^{nn}\sim \mathcal{F}$ and $F^{nn}u_n^2\sim \mathcal{F}|Du|^2$  is the leading term.
}}

At $x_0$, we  have
\begin{align}\label{1stdc}
0=G_i=\frac{|Dw|^2_i}{|Dw|^2}+(M-u)^{-1}u_{i}+a_0 d_i.
\end{align}

By the maximum principle, at $x_0$, we have 
\begin{align}
0\le& \sum_{i,j}F^{ij}G_{ij}=\sum_{i,j}F^{ij}(\log|Dw|^2)_{ij}\notag\\
&+(M-u)^{-2}\sum_{i}F^{ii}u_{i}^2+(M-u)^{-1}\sum_{i}F^{ii}\lambda_i+a_0\sum_{i}F^{ii}d_{ii}.
\end{align}
By  the estimate for $\sum\limits_{i,j}F^{ij}(\log|Dw|^2)_{ij}$ in Lemma \ref{FijlogDw} and  the first derivative condition \eqref{1stdc}, we have
\begin{align}
0\ge& \sum_{i,j}F^{ij}G_{ij}
\ge
-C (d+|Du|^{-1}) \sum_{i}{F^{ii}}|\lambda_i|-C(d|Du|^2+|Du|)\mathcal{F}\notag\\
&-\frac{1}{2}\sum_{i,j}F^{ij}\frac{(|Dw|^2)_{i}(|Dw|^2)_{j}}{|Dw|^4}+(M-u)^{-2}\sum_{i}F^{ii}u_{i}^2\\
&+(M-u)^{-1}\sum_{i}F^{ii}\lambda_i+a_0\sum_{i}F^{ii}d_{ii}
\notag\\
=&-C (d+|Du|^{-1})\sum_{i}{F^{ii}}|\lambda_i|-C(d|Du|^2+|Du|)\mathcal{F}+\frac{1}{2}(M-u)^{-2}\sum_{i}F^{ii}u_{i}^2\notag\\
&+(M-u)^{-1}\sum_{i}F^{ii}\lambda_i-(M-u)^{-1}a_0 \sum\limits_{i}F^{ii} u_id_i+a_0\sum_{i}F^{ii}(d_{ii}-\frac{1}{2} a_0d_i^2)\notag\\
\ge& -C(d+|Du|^{-1}) \sum_{i}{F^{ii}}|\lambda_i|-C(d|Du|^2+|Du|+\alpha_0)\mathcal{F}\notag\\
&+\frac{1}{4}(M-u)^{-2}\sum_{i}F^{ii}u_{i}^2+(M-u)^{-1}\sum_{i}F^{ii}\lambda_i,\label{FW2023122}
\end{align}
where in the last inequality we use the Cauchy inequality.

%By Lemma \ref{WarrenYuan2010AJM}, we know $\sum_{i}F^{ii}\lambda_i\ge 0$.

Since there  exists $i_0$ such that 
$
|u_{i_0}|^2\ge \frac{1}{{n}}|Du|^2$, without loss of generality, we may assume \begin{align}u_n\ge \frac{1}{\sqrt{n}}|Du|.
\end{align}
Note that $1+d\varphi \in (\frac{2}{3},\frac{4}{3})$ and assuming $u_n>>1$, we have 
\begin{align}
w_n=(1+\varphi d)u_{n}+O(1)\in (\frac{1}{2}u_n,2u_n).
\end{align}
Based on this inequality, we claim that:
\begin{align}\label{unnun}
C(n)^{-1}u^2_{n}\le -u_{nn}\le C(n) u^2_{n}.
\end{align}
To demonstrate this, we consider \eqref{1stdc}:
\begin{align}
w_n w_{nn}=&-\frac{u_n|Dw|^2}{2(M-u)}-\sum_{k=1
}^{n-1}w_kw_{kn}-\frac{a_0}{2}d_n|Dw|^2\notag\\
=&-(\frac{1}{2}+O(d))\frac{u_n|Dw|^2}{M-u}+O(|Du|^2),
\end{align}
where we utilize $w_{kn}=O(d|Du|^2+|Du|)$ and $u_{kn}=0, \forall k<n$.
The claim then follows from the above analysis and $u_{nn}=\frac{w_{nn}}{1+\varphi_{u}d}+O(d|Du|^2+|Du|)$.

Given that $\lambda_n=u_{nn}<0$ and $\Theta\ge (n-2)\frac{\pi}{2}$, we can establish the inequality:
\begin{equation}
\lambda_i\ge |u_{nn}|, \quad \forall \quad i<n. \label{lambda_i>u_nn}
\end{equation}
If $\Theta\geq (n-2)\frac{\pi}{2}+\delta$, then from Lemma \ref{WWY}  
\begin{equation*}
 -\lambda_n\le C(\delta)\max|f|.   
\end{equation*}
Thus we have by  \eqref{unnun} that 
\begin{equation*}
  |Du|^2(x_0) \leq C(\delta,f,n). 
\end{equation*}
Without loss of generality, let's assume $\theta=(n-2)\frac{\pi}{2}$.
By combining \eqref{unnun} and \eqref{lambda_i>u_nn}, we obtain:
\begin{align}
F^{nn}=&\frac{f}{f^2+\lambda_n^2}\ge \frac{1}{n}\sum_{i}\frac{f}{f^2+\lambda_i^2}=\frac{\mathcal{F}}{n},\notag\ \\
\sum_{i}F^{ii}|\lambda_{i}|\le&C\sum_{i}|\lambda_i|^{-1}\le C|Du|^{-2}.
\label{Fiilambdai23122}
\end{align}
This leads to:
\begin{align}
\frac{1}{2}\sum_{i}F^{ii}u_{i}^2\ge \frac{1}{2}F^{nn}u_n^2\ge \frac{\mathcal{F}|Du|^2}{2n^2}.\label{Fiiui23122}
\end{align}

We observe from \eqref{unnun} that
\begin{equation}
    F^{nn}u_n^2\geq \frac{f u^2_n}{f^2+u^2_{nn}} \geq c|Du|^{-2}. \label{F^nnu_n^2} 
\end{equation}

Inserting  inequalities \eqref{Fiiui23122}, \eqref{Fiilambdai23122} and Lemma \ref{WY2010} into \eqref{FW2023122},
we obtain 
\begin{align}
0\ge\sum_{i,j}F^{ij}G_{ij}\ge&\Big(\frac{1}{2n^2}|Du|^2-C(d|Du|^2+|Du|)\Big) \mathcal F\notag\\
&+\frac{1}{2}\sum_{i}F^{ii}u_i^2-C(d|Du|^{-2}+|Du|^{-3}).
\end{align}
Due to \eqref{F^nnu_n^2}, the last two terms are positive provided $\mu$ is small enough. Therefore, we get the uniform estimate.

\end{proof}
\subsection{Uniform gradient estimate for the classical Neumann problem}
We will show the uniform gradient estimate which is independent of the $C^0$ norm of the solution. Here the uniformly convexity condition is crucial. 
\begin{lemma} \label{UniformGradient}
Let $u$ be  a $C^3$ solution of the  problem \eqref{aslte11}. For suffciant small constant $\epsilon$, there exists a uniform constant $C$ depending on $|\phi|_{C^3},|f|_{C^1},|f^{-1}|_{L^\infty},n,|\partial \Omega|_{C^2}, |u|_{C^0}$ and uniformly convexity of  $\partial \Omega$ such that
\begin{align}
\max_{\overline\Omega}|Du|\le C.
\end{align}
\end{lemma}
\begin{remark}
The new problematic term is $|Du|^{-1}\sum_{i,j}{F^{ij}}u_{ij}$, while the favorable term is $\mathcal{F}$. The crucial observation is that when $\Theta=(n-2)\frac{\pi}{2}$, we can establish $0\le \sum_{i,j}{F^{ii}}u_{ij}\le C\mathcal{F}$.
\end{remark}
\begin{proof}
For simplicity we assume $0\in \Omega$ and consider the following function 
\begin{align}
    P=\log |Dw|^2+\frac{b}{2}|x|^2,
\end{align}
where $w=(1+\varepsilon h)u-\phi h$ and $h$ is the defining function of $\Omega$ satisfying $h_{\nu}=|Dh|=1$ on $\partial\Omega$. 
Assume $P(x_0)=\max_{\overline \Omega} P(x)$.\\
{\bf{Case 1:}} $x_0\in \partial \Omega$.\\
We use
\begin{align}
0\le P_{\nu}(x_0)
\end{align}
and refer to Proposition 5 in \cite{QiuXia19} to obtain the estimate from the uniform convexity of $\partial \Omega$ provided that $b$ is small.\\
{\bf{Case 2: $x_0\in \Omega$.}}\\
We choose the coordinate such that $u_{ij}(x_0)=\lambda_i\delta_{ij}$.
We assume $|Du|(x_0)>>1$. Without loss of generality, we may assume at $x_0$
\begin{align*}
u_{n}\ge \frac{1}{\sqrt{n}}|Du|.
\end{align*}
Firstly, we have
\begin{align}\label{Pi}
0=P_i=\frac{2w_kw_{ki}}{|Dw|^2}+bx_i.
\end{align}
Recall that \begin{align*}
w_k=&(1+\varepsilon h)u_k+ h_k \varepsilon u-\phi_k h - \phi h_k, \\
w_{ki}=&(1+\varepsilon h)u_{ki}+\varepsilon(h_iu_k+h_ku_i)+h_{ki}\varepsilon u-\phi_{ki}h - \phi_k h_i -\phi_i h_k -\phi h_{ki},\\
w_{kij}=& (1+\epsilon h) u_{kij}+\epsilon (h_j u_{ki}+h_i u_{kj}+h_{k} u_{ij}+h_{ij}u_k+h_{kj}u_i+h_{ki} u_j)\\
& +\epsilon h_{kij} u-\phi_{kij} h - \phi_{ki}h_j -\phi_{kj} h_i -\phi_k h_{ij}\\
&-\phi_{ij} h_k - \phi_i h_{kj} - \phi_j h_{ki} -\phi h_{kij}.
\end{align*}
Combing the above with \eqref{Pi} and noting that  $\varepsilon u$ is uniformly bounded, we have at $x_0$
\begin{align}
w_{n}\sim
u_{n}&\sim |Du|,\nonumber \\
|u_{nn}|\le& C|Du|.\label{lambdanleq}
\end{align}
Similar as Lemma \ref{FijlogDw}, we obtain
\begin{align}
0\ge& \sum_{i,j}F^{ij}P_{ij}
\ge
\sum_{i,j}F^{ij}(\log|Dw|^2)_{ij}+b\mathcal{F}\notag\\
\ge& -C|Du|^{-1} 
\Big|\sum_{i}{F^{ii}}\lambda_i\Big|-C\epsilon |Du|^{-1} 
\sum_{i}{F^{ii}}|\lambda_i|\\
&  +\Big(b(1-Cb)-C|Du|^{-1}-\varepsilon\Big)\mathcal{F} . \label{F^ijP_ijUniformlyGradient}
\end{align}
If $\lambda_n \geq 0$, we have 
\begin{equation*}
    \sum_{i}{F^{ii}}|\lambda_i| = \sum_{i}{F^{ii}} \lambda_i.
\end{equation*}
Or $\lambda_n <0 $, we have
\begin{equation*}
    \sum_{i}{F^{ii}}|\lambda_i| = \sum_{i}{F^{ii}}\lambda_i + 2F^{nn} |\lambda_n|.
\end{equation*}
Regardless of the sign of $\lambda_n$, equation \eqref{F^ijP_ijUniformlyGradient} can be expressed as:
\begin{equation}
    0\geq -C|Du|^{-1} 
\Big|\sum_{i}{F^{ii}}\lambda_i\Big| -2 C\epsilon |Du|^{-1} F^{nn} |\lambda_n| +\frac{b}{2}\mathcal{F}. \label{F^ijP_ijUniformlyGradient2}
\end{equation}
If $\Theta\geq (n-2)\frac{\pi}{2}+\delta$, from Lemma \ref{WWY} we can deduce that
\begin{equation*}
  \mathcal{F}\ge \frac{f}{f^2+ \min\limits_{1\le i\le n}\lambda^2_i} \geq c_0 (f,\delta) .  
\end{equation*}
We know $\sum\limits_{i}{F^{ii}}|\lambda_i| \leq C$, then we obtain 
\begin{equation*}
    |Du|\le C.
\end{equation*}
In the following, we  assume $\Theta=(n-2)\frac{\pi}{2}$. By Lemma \ref{WY2010}, we know that
\begin{equation*}
    \sum_i F^{ii}\lambda_i\ge 0.
\end{equation*}
We need prove $\sum\limits_i F^{ii}\lambda_i$ can be controlled by $\mathcal{F}$.
Without loss of generality, let's assume that $\lambda_1 \geq \lambda_2 \geq \cdots \geq \lambda_n$.\\

Case 1: $\lvert \lambda_n\rvert <C$.
We have $\mathcal{F}\ge c_0$, and then it is not hard to derive the estimate 
\begin{equation*}
  |Du|\le C.  
\end{equation*}

Case 2: $\lvert \lambda_n \rvert \geq C$. \\
Due to $\lambda_{n-1}+\lambda_n \geq 0$, we have 
\begin{equation*}
    \mathcal{F}\geq \frac{f}{f^2+\lambda^2_n} \geq c \frac{\lambda_1 \lambda_2 \cdots \lambda_{n-2}}{V}, 
\end{equation*}
where $V=\prod\limits_{i=1}^{n}\sqrt{f^2+\lambda_i^2}$ and $c>0$ depends only on $f$.
When $\Theta =\frac{(n-2)\pi}{2}$, our equation is 
\begin{equation*}
\sigma_{n-1} (\frac{D^2 u}{f}) - \sigma_{n-3} (\frac{D^2 u}{f}) +\cdots = 0.    
\end{equation*}
We know that 
\begin{equation*}
    F^{ij} = \frac{\sigma^{ij}_{n-1} - \sigma^{ij}_{n-3}  +\cdots}{V}.
\end{equation*}
Thus using the equation once to cancel the term with $\sigma_{n-1}$, we have
\begin{equation*}
    \sum_i F^{ii} \lambda_i \leq C(f) \frac{\sum_{i\geq 3} \lvert \sigma_{n-i} \rvert }{V}.
\end{equation*}

Because 
\begin{equation*}
    \lambda_1 \lambda_2 \cdots \lambda_{n-2} \geq c\sum_{i\geq 3} \lvert \sigma_{n-i} \rvert.
\end{equation*}
We have the inequality
\begin{equation}
    \mathcal{F} \geq  \sum_i  F^{ii} \lambda_i. \label{mathcalF}
\end{equation}
For the second term of \eqref{F^ijP_ijUniformlyGradient2}, we have from   \eqref{lambdanleq} 
\begin{equation}\label{epsilonmathcalF}
    -2 C\epsilon |Du|^{-1} F^{nn} |\lambda_n|  \geq   
 -2 C\epsilon \mathcal{F}.
\end{equation} 
Thus by \eqref{mathcalF} and \eqref{epsilonmathcalF} we estimate inequality \eqref{F^ijP_ijUniformlyGradient2} as following
\begin{equation*}
    0 \geq -C |Du|^{-1} \mathcal{F} - 2C  \epsilon \mathcal{F} + \frac{b}{2} \mathcal{F}.
\end{equation*}
Then we have the estimate
\begin{equation*}
    |Du| \leq C.
\end{equation*}

\end{proof}

\section{Second order estimates}

\subsection{Boundary double normal derivative estimate} \label{DDN}
We  prove the double normal derivative estimate directly. In particular, we give a direct proof of the double normal derivative estimates for the 2-Hessian equation in dimension 3.
\begin{theorem}\label{normal}
Let 
$u$ be a $C^4$ solution of problem \eqref{GSLE}
There exits a positive constant $C$ depending only on $n, f, \varphi, \Omega$ such that
\begin{align}
\max_{\partial \Omega} |u_{\nu\nu}|\le C.
\end{align}
\end{theorem}

We consider 
\begin{equation*}
    \overline P=u_{\nu} - \varphi-\frac{1}{2}(u_{\nu} - \varphi)^2-B_0 h \quad in \quad \overline\Omega_{\mu},
\end{equation*}
where $h=-d(x)+d^2(x)$ and satisfies
\begin{align*}
Dh=&\nu\quad \text{on}\quad \partial \Omega,\\
D^2 h\ge& \kappa_0 I \quad\text{in}\quad \Omega_{\mu}.
\end{align*}

\begin{lemma} \label{lem4.1} There exist a positive constant $B_0$ large enough such that  $\overline  P$ only attains its minimum  on  $\partial\Omega$.
\end{lemma}
\begin{proof}
Suppose $\overline P$ attains its minimum at $x_0\in \Omega_{\mu}$. 
Assume $\{D^2 u(x_0)\}=\{\lambda_i\delta_{ij}\}$ with $\lambda_1\ge \lambda_2\ge \cdots\ge \lambda_n$. \\
By direct calculation, we have 
\begin{align*}
(u_{\nu}-\varphi)_{i}=\ &u_{ki}\nu^{k}+u_{k}\nu^k_{i}-\varphi_{x_i}-\varphi_{u}u_{i},\\
(u_{\nu}-\varphi)_{ij}=\ &u_{kij}\nu^{k}+u_{ki}\nu^k_{j}+u_{kj}\nu^k_{i}+u_{k}\nu^k_{ij}-\varphi_{x_ix_j}-\varphi_{x_iu}u_{j}-\varphi_{ux_j}u_{i}-\varphi_{uu}u_{i}u_{j}-\varphi_{u}u_{ij}.
\end{align*}
Then by \eqref{diff1time}, we get
\begin{align}
\sum_{i,j}F^{ij}(u_{\nu}-\varphi)_{ij}=&\sum_{ij}F^{ij}u_{kij}\nu^{k}+2\sum_{i,j,k}F^{ij}u_{ki}\nu^k_{j}-\varphi_{u}\sum_{i,j}F^{ij}u_{ij}\notag\\
&+
\sum_{i,j}F^{ij}(\sum_{k}\nu^k_{ij}u_{k}-\varphi_{x_ix_j}-2\varphi_{x_iu}u_{j}-\varphi_{uu}u_{i}u_{j}),\notag\\
=&<D\log f,\nu>\sum_{i}F^{ii}\lambda_i +2F^{ii}\lambda_i\nu^{i}_{i}-\varphi_{u}F^{ii}\lambda_i+O(
\mathcal{F}).\label{Funu11}
\end{align}
We also have
\begin{align}\label{Funu1}
\sum_{i}F^{ii}\Big((u_{\nu}-\varphi)_{i}\Big)^2\ge \sum_{i}F^{ii}\lambda_i^2(\nu^i)^2-C\sum_{i}F^{ii}|\lambda_i|+O(\mathcal{F}).
\end{align}

By the maximum principle and \eqref{Funu11}, \eqref{Funu1}, we have 
\begin{align}\label{24126}
0\le\sum_{i,j} F^{ij}\overline P_{ij}
=& (1-u_{\nu}+\varphi)\sum_{i}F^{ii}(u_{\nu}-\varphi)_{ii}-\sum_{i}F^{ii}[(u_{\nu}-\varphi)_i]^2 - B_0 \sum_{i}F^{ii}h_{ii}\notag\\
\le& C\sum_{i}F^{ii}|\lambda_i|- (\kappa_0 B_0-C)\sum_{i}F^{ii}-\frac{1}{2}\sum_{i}F^{ii}\lambda_i^2(\nu^i)^2.
\end{align}

\emph{We claim there exists a positive constant $A_0$ such that}
\begin{align*}
\sum_{i}F^{ii}|\lambda_i|\le  A_0\sum_{i}F^{ii} + \frac{1}{2C}\sum_{i}F^{ii}\lambda_i^2(\nu^i)^2.
\end{align*}
Combining \eqref{24126} with the claim and choosing $B_0$ large, we arrive at the following contradiction
\begin{align*}
0\le&\sum_{i,j} F^{ij}\bar P_{ij}
\le -(\kappa_0 B_0-CA_0-C)\sum_{i}F^{ii}<0.
\end{align*}
Thus $\overline{P}$ attains its minimum  on $\partial \Omega_{\mu}$. Since $\overline P|_{\partial\Omega}=0$ and $\overline P|_{\partial \Omega_{\mu}\cap \Omega}\ge -C+\frac{1}{2}B_0\mu>0$ if we choose $B_0$ large enough, we conclude $\overline{P}$ attains its minimum  on $\partial \Omega$. 

Now we prove the claim. We divide two cases to get the proof.\\
%%{\bf{Case1:}} $|\lambda_n|\ge c_0 M^{\frac{1}{4}}$\\
%We have  $|\lambda_i|\ge c_0 M^{\frac{1}{4}}$ and if we choose $M$ large enough, then $\sum_{i}F^{ii}\lambda_i^2\nu_i^2\ge \frac{n}{2}f$. We also have 
%$\sum_{i}F^{ii}|\lambda_i|\le nfM^{-\frac{1}{4}}$. 
%Thus we have 
%\begin{align}
%\sum_{i}F^{ii}|\lambda_i|\le a\sum_{i}F^{ii}\lambda_i^2\nu_i^2.
%\end{align}
%%
{\bf{Case 1:} $|\lambda_n|\ge C_0:=2nC(|f|_{C^0}+1)$}.\\
When $\Theta \geq \frac{(n-2)\pi}{2}$, we know that $\lambda_i\geq |\lambda_n|$ for $\forall i>n$.
Thus we have $|\lambda_i| \geq C_0 $ for $\forall i$. 
Then we can observe:
\begin{equation*}
    F^{ii}\lambda_i^2 = \frac{f\lambda_i^2}{f^2+\lambda_i^2} \geq \frac{fC_0^2}{f^2+C_0^2} \geq \frac{f}{2}, \quad \forall 1\le i\le n,
\end{equation*}
where we've used \(C_0 > |f|_{C^0}\).
 
Thus we get $$\sum_{i}F^{ii}\lambda_i^2(\nu^i)^2\ge\sum_{i}\frac{f}{2}(\nu^i)^2 = \frac{f}{2}.$$ 
Then we have
\begin{align*}
\sum_{i}F^{ii}|\lambda_i|\le \sum_i \frac{f}{|\lambda_i|} \leq  nfC_0^{-1}\le \frac{1}{2C}\sum_{i}F^{ii}\lambda_i^2(\nu^i)^2.
\end{align*}
{\bf{Case 2:}} $|\lambda_n|\le C_0$. \\
In this case,  $\sum\limits_{i}F^{ii}\ge \frac{f}{f^2+\lambda_n^2}\ge\frac{f}{f^2+C_0^2}.$ 
Then we get  
\begin{align*}
\sum_{i}F^{ii}|\lambda_i|\le  \sum\limits_i \frac{f|\lambda_i|}{f^2 + \lambda_i^2} \leq   \frac{n}{2} \le A_0\sum_{i}F^{ii},
\end{align*}
where we choose $A_0$ large. So we proved the claim.

\end{proof}

Next we consider the function
\begin{equation*}
\underline{P}=u_\nu -\varphi+\frac{1}{2}(u_\nu -\varphi)^2+B_0 h \quad in \quad \overline\Omega_{\mu}.
\end{equation*}
\begin{lemma} There exist a positive constant $B_0$ large enough such that  $\underline{P}$ only attains its minimum only on  $\partial\Omega$.
\end{lemma}
\begin{proof}
Suppose $\underline{P}$ attains its maximum at $x_0\in \Omega_{\mu}$. 
Assume $\{D^2 u(x_0)\}=\{\lambda_i\delta_{ij}\}$ with $\lambda_1\ge \lambda_2\ge \cdots\ge \lambda_n$. \\
Similarly as Lemma \ref{lem4.1}, we get the following contradiction
\begin{align}
0\ge F^{ij}\underline{P}_{ij}\ge -C\sum_{i}F^{ii}|\lambda_i|+  (\kappa_0 B_0-C)\sum_{i}F^{ii}+\frac{1}{2}\sum_{i}F^{ii}\lambda_i^2\nu_i^2
>0.
\end{align}
Thus $\underline{P}$ attains its maximum only on $\partial \Omega_{\mu}$.
Since $\underline P|_{\partial\Omega}=0$ and $\underline P|_{\partial \Omega_{\mu}\cap \Omega}\le C-\frac{1}{2}B_0\mu<0$ if we choose $B_0$ large enough, we conclude $\underline{P}$ attains its maximum  on $\partial \Omega$. 
\end{proof}
We use the above two lemmas to prove the double normal derivative estimates.

\emph{Proof of Theorem \ref{normal}}
Since $\bar P$ attains its minimum $0$ at any $x\in \partial \Omega$. We get for any $x\in \partial \Omega$,
\begin{align*}
0\ge& \bar{P}_{\nu}(x)=u_{\nu\nu}(x)-\varphi_{x_k}(x,u(x))\nu^k(x)-\varphi_{u}(x, u(x))u_{\nu}(x)-B_0h_{\nu}\\
=&u_{\nu\nu}(x)-B_0-\varphi_{x_k}(x,u(x))\nu^k(x)-\varphi_{u}(x, u(x))\varphi(x,u(x)).
\end{align*}
This gives the upper bound of $u_{\nu\nu}$.
Similarly, we get the lower bound of $u_{\nu\nu}$ since $\underline{P}$ attains its maximum $0$ at any $x\in \partial \Omega$.
Therefore, we have the uniform double normal derivative estimates.

\subsection{Global second order estimates}
We use a similar auxiliary function as introduced by Lions-Trudinger-Urbas \cite{LTU1986CPAM} to reduce the second order estimate  to the boundary double normal derivative.

\begin{theorem}
Let 
$u$ be a $C^4$ solution of problem \eqref{GSLE}. Then we have
\begin{align}
\max_{\overline\Omega} |D^2 u|\le C(1+\max_{\partial\Omega} |u_{\nu\nu}|),
\end{align}
where $C$ is a positive constant.
\end{theorem}
\begin{remark}
    We remark that during the proof we only use the strict convexity of $\Omega$ and we do not use the positive lower bound of $-\varphi_u$ and thus the estimate here can be applied to the classical problem.
\end{remark}
\begin{proof}
We consider the following function
\begin{align}
V(x, \xi)=u_{\xi\xi}-v(x, \xi)+ \frac{|Du|^2}{2}+B\frac{|x|^2}{2},
\end{align}
where $v(x, \xi)=2<\xi,\nu><\xi', D \varphi-Du-u_kD\nu^{k}>=a^k(x)u_k+b(x)$, $\xi'=\xi-<\xi,\nu>\nu$ and $B$ is a positive constant to be determined later.

The estimate is equivalent to prove an uniform upper bound of $u_{\xi\xi}$.

We want to show $V$ attains its maximum on the boundary $\partial \Omega$ by choosing $ B$ large enough. Indeed, if there exists a point $x_0\in \Omega$ such that $V(x_0)=\max_{\overline \Omega} V$, we choose coordinates such that $D^2 u(x_0)=\{\lambda_i\delta_{ij}\}$. Then at $x_0$ we have
\begin{align}
F^{ij}=\frac{f\delta_{ij}}{f^2+\lambda_i^2},
\end{align}
and 
\begin{align}
F^{ij,kl}=\left\{
\begin{aligned}
-&\frac{f(\lambda_i+\lambda_j)}{(f^2+\lambda_i^2)(f^2+\lambda_j^2)}, i=l,k=j,\\
&0,\qquad  \quad\qquad\qquad \text{otherwise}.
\end{aligned}
\right.
\end{align}
By maximum principle and direct calculation, we have 
\begin{align}
0\ge \sum_{i,j}F^{ij}V_{ij}=&\sum_{i,j}F^{ij}u_{ij\xi\xi}-\sum_{k}(a^k+u_k)\sum_{i,j}F^{ij}u_{ijk}-2\sum_{i}a^i_{i}F^{ii}u_{ii}-F^{ii}D_{ii}a^{k}u_{k}\notag\\
&+F^{ij}b_{ij}+\sum_{i=1}^nF^{ii}u_{ii}^2+B\mathcal{F}\label{FijVij}.
\end{align}
By \eqref{deri1} and \eqref{deri2} in  Lemma \ref{deri12}, we have
\begin{align}
F^{ij}u_{ij\xi\xi}-(a^k+ u_k))F^{ij}u_{ijk}\ge -C_f \sum_{i=1}^n{F^{ii}|\lambda_i|},
\end{align}
where $C_f=2(1+|f^{-1}|_{C^0})|D\log f|^2_{C^0}+|f^{-1}|_{C^0}|D^2f|_{C^0}+\sum_{k}|a^{k}|_{C^0}+ |Du|_{C^{0}}$. Inserting the above into \eqref{FijVij}, we get
\begin{align}
0\ge \sum_{i,j}F^{ij}V_{ij}\ge\sum_{i=1}^n F^{ii}(\lambda_i^2-C|\lambda_i|+B).\label{FijVij2}
\end{align}
If we choose $B=2C^2$, then we get a contradiction from \eqref{FijVij2}. Thus $V$ attains its maximum on the boundary $\partial \Omega$. 

Assume $V(x_0,\xi_0)=\max\limits_{\overline \Omega\times \mathcal{S}^{n-1}} V(x, \xi)$, where $\mathcal{S}^{n-1}$ is the unit sphere in $\mathbb{R}^{n}$. By the above proof, we know $x_0\in \partial \Omega$. 

If $\xi_0$ is the normal direction. By the double normal estimate, we get the proof. 

If $\xi_0$ is non-tangential i.e. $<\xi_0, \nu>\neq 0$. By the decomposition, $\xi_0=a \nu+b \tau$, where $a=<\xi_0, \nu(x_0)>$ and $b=<\xi_0, \tau>$ and $\tau$ is the unit tangential part of $\xi_0$. Then  by $v(x_0, \xi_0)=a^2 v(x_0, \tau)+ b^2  v(x_0,\nu)$, we have $v(x_0, \xi_0)\le v(x_0, \nu)\le C$ and thus $u_{\xi_0\xi_0} \le C$, where we used the double normal estimate.

If $\xi_0=e_1$ is the tangential direction, we refer to \cite{MaQiu2019CMP} for the details of deriving the following inequalities \eqref{V1} and \eqref{V2}.
First, we have an inequality at the boundary point $x_0$:
\begin{equation} \label{V1}
0\leq V_\nu (x_0,e_1)= - u_{11n}(x_0) +C.
\end{equation}
On the other hand,  differentiating $u_\nu = \varphi$ along the tangential  direction twice, considering $\varphi_u \leq 0$ and the uniform convexity of $\partial \Omega$, we have 
\begin{equation}\label{V2}
   -u_{11n}\leq -2 u_{11} \kappa_0+C.
\end{equation}
Combining \eqref{V1} and \eqref{V2}, we obtain
\begin{equation*}
    u_{11}(x_0)\le C(\kappa_0, |\varphi|_{C^{2}}, |u|_{C^1}).
\end{equation*}

\end{proof}

\section{Proofs of Theorem \ref{sltethm1} and Theorem \ref{sltethm2}}
In this section, we use  a priori estimates proved in the previous sections to get the existence.
Let $\Omega_t=t\Omega+(1-t)B_1$. 
Consider the following problem 
\begin{align}\label{existence}
\left\{
\begin{aligned}
    \sum_{i=1}^n \arctan \frac{\lambda_i (D^2 u^t)}{tf+1-t} = \Theta \quad in \quad \Omega_t, \\
    u^t_\nu = -u^t+ t\phi\quad on \quad \partial \Omega_t.
    \end{aligned}
    \right.
\end{align}
%Let $\kappa_0=\min_{\partial \Omega}\min_{1\le i\le n}\kappa_i$. Then $\kappa_i(\partial\Omega_t)\ge \min\{1, \kappa_0\}$.%
Since $\Omega$ is strictly convex, $\Omega_t$ is strictly convex. By Theorem \ref{sltethm1} and  Evans-Krylov-Safonov theory as in \cite{LiebermanTrudinger}, there exists a uniform constant $C$ depending on $n$, $\Omega$, $f$, $\phi$ such that
\begin{align}\label{241261344}
|u^t|_{C^{2,\alpha}}\le C.
\end{align}
Define the set $\mathcal{I}=\{t\in[0,1]: \ \text{problem}\  \eqref{existence} \ \text{has a $C^{2,\alpha}$ solution} \}$.
When $t=0$, $\Omega_0=B_1$, there exists a unique smooth solution.
The openness of $\mathcal{I}$ follows from the implicit function theorem. The closeness follows from the $C^{2,\alpha}$ estimates \eqref{241261344}. Then $\mathcal{I}=[0,1]$ and thus we obtain Theorem
\ref{sltethm1}. 

For Theorem \ref{sltethm2}, we first consider the following approximating equation.
\begin{align}
\left\{
\begin{aligned}
    \sum_{i=1}^n \arctan \frac{\lambda_i (D^2 u^{\varepsilon})}{f(x)} = \Theta \quad in \quad \Omega, \\
    u^\varepsilon_\nu =  -\varepsilon u^\varepsilon+\phi(x)\quad on \quad \partial \Omega.
    \end{aligned}
    \right.
\end{align}
By Theorem \ref{sltethm1}, there exists a unique smooth solution $u^\varepsilon$. Due to Lemma \ref{UniformGradient}, we have
\begin{equation*}
 |\nabla u^\varepsilon|\leq C   
\end{equation*}
independent of $\varepsilon$. Then by $C^0$ estimate \eqref{epsilonU}, there is
a constant $\lambda$, such that
\begin{equation*}
    -\varepsilon u^\varepsilon \rightarrow\lambda \quad \text{ as } \quad \varepsilon \rightarrow 0.
\end{equation*}
So we solve the following classical Neumann equation

\begin{align}
\left\{
\begin{aligned}
    \sum_{i=1}^n \arctan \frac{\lambda_i (D^2 u)}{f(x)} = \Theta \quad in \quad \Omega, \\
    u_\nu =  \lambda+\phi(x)\quad on \quad \partial \Omega.
    \end{aligned}
    \right.
\end{align}
Then we prove uniqueness. Suppose problem (\ref{slte11}) has
two pairs of solutions $(\lambda,u)$ and $(\mu,v)$. Let $a^{ij}=\int_{0}^{1}F^{ij}[(1-t)\frac{D^{2}v}{f}+t\frac{D^{2}u}{f}]dt$,
and $u-v$ satisfies 
\begin{equation}
\begin{cases}
\int^1_0 \frac{d}{dt} \arctan [(1-t)\frac{D^{2}v}{f}+t\frac{D^{2}u}{f}] dt =a^{ij}\frac{(u-v)_{ij}}{f}=0,\\
(u-v)_{\nu}=\lambda-\mu.
\end{cases}
\end{equation}
So $u-v$ attains its maximum and minimum at the boundary. This implies that $\lambda=\mu$.
Finally, applying the Hopf lemma from \cite[Theorem 3.6]{gilbarg2015elliptic}, we deduce
$u-v=c$.

\textbf{Acknowledgements}
Both authors are grateful to Professor Xi-Nan Ma for his interest in this paper. 
 The first named author is partially supported by CAS Project for Young Scientists in Basic Research, Grant No.YSBR-031 and grants from the Research Grants Council of the Hong Kong Special Administrative region, China [Project No: CUHK 14300819].

%\textbf{\large Declaration}\\
%\textbf{Conflict of interest.} 
%The authors state that there is no conflict of interest.
\bibliographystyle{acm}
\bibliography{Bib}

\begin{thebibliography}{10}

\bibitem{BhattacharyaCVPDE21}
{\sc Bhattacharya, A.}
\newblock Hessian estimates for {L}agrangian mean curvature equation.
\newblock {\em Calc. Var. Partial Differential Equations 60}, 6 (2021), Paper
  No. 224, 23.

\bibitem{bhattacharya2022note}
{\sc Bhattacharya, A.}
\newblock A note on the two-dimensional lagrangian mean curvature equation.
\newblock {\em Pacific Journal of Mathematics 318}, 1 (2022), 43--50.

\bibitem{bhattacharya2022gradient}
{\sc Bhattacharya, A., Mooney, C., and Shankar, R.}
\newblock Gradient estimates for the lagrangian mean curvature equation with
  critical and supercritical phase.
\newblock {\em arXiv preprint arXiv:2205.13096\/} (2022).

\bibitem{bhattacharya2020optimal}
{\sc Bhattacharya, A., and Shankar, R.}
\newblock Optimal regularity for lagrangian mean curvature type equations.
\newblock {\em arXiv preprint arXiv:2009.04613\/} (2020).

\bibitem{BhattacharyaShankarCrelle}
{\sc Bhattacharya, A., and Shankar, R.}
\newblock Regularity for convex viscosity solutions of {L}agrangian mean
  curvature equation.
\newblock {\em J. Reine Angew. Math. 803\/} (2023), 219--232.

\bibitem{Brendle21}
{\sc Brendle, S.}
\newblock The isoperimetric inequality for a minimal submanifold in {E}uclidean
  space.
\newblock {\em J. Amer. Math. Soc. 34}, 2 (2021), 595--603.

\bibitem{BrendleWarrenJDG10}
{\sc Brendle, S., and Warren, M.}
\newblock A boundary value problem for minimal {L}agrangian graphs.
\newblock {\em J. Differential Geom. 84}, 2 (2010), 267--287.

\bibitem{Cabre08}
{\sc Cabr\'{e}, X.}
\newblock Elliptic {PDE}'s in probability and geometry: symmetry and regularity
  of solutions.
\newblock {\em Discrete Contin. Dyn. Syst. 20}, 3 (2008), 425--457.

\bibitem{CNS-1}
{\sc Caffarelli, L., Nirenberg, L., and Spruck, J.}
\newblock The {D}irichlet problem for nonlinear second-order elliptic
  equations. {I}. {M}onge-{A}mp\`ere equation.
\newblock {\em Comm. Pure Appl. Math. 37}, 3 (1984), 369--402.

\bibitem{CNS-3}
{\sc Caffarelli, L., Nirenberg, L., and Spruck, J.}
\newblock The {D}irichlet problem for nonlinear second-order elliptic
  equations. {III}. {F}unctions of the eigenvalues of the {H}essian.
\newblock {\em Acta Math. 155}, 3-4 (1985), 261--301.

\bibitem{ChenMaWei2019ATA}
{\sc Chen, C., Ma, X., and Wei, W.}
\newblock The {N}eumann problem of complex special {L}agrangian equations with
  supercritical phase.
\newblock {\em Anal. Theory Appl. 35}, 2 (2019), 144--162.

\bibitem{ChenMaWei2019JDE}
{\sc Chen, C., Ma, X., and Wei, W.}
\newblock The {N}eumann problem of special {L}agrangian equations with
  supercritical phase.
\newblock {\em J. Differential Equations 267}, 9 (2019), 5388--5409.

\bibitem{ChenZhang2021BMS}
{\sc Chen, C., and Zhang, D.}
\newblock The {N}eumann problem of {H}essian quotient equations.
\newblock {\em Bull. Math. Sci. 11}, 1 (2021), Paper No. 2050018, 26.

\bibitem{ChenMaZhang2021ActaSinica}
{\sc Chen, C.~Q., Ma, X.~N., and Zhang, D.~K.}
\newblock The {N}eumann problem for parabolic {H}essian quotient equations.
\newblock {\em Acta Math. Sin. (Engl. Ser.) 37}, 9 (2021), 1313--1348.

\bibitem{ChenGao2021}
{\sc Chen, G.}
\newblock The {J}-equation and the supercritical deformed
  {H}ermitian-{Y}ang-{M}ills equation.
\newblock {\em Invent. Math. 225}, 2 (2021), 529--602.

\bibitem{ChenShankarYuan}
{\sc Chen, J., Shankar, R., and Yuan, Y.}
\newblock Regularity for convex viscosity solutions of special {L}agrangian
  equation.
\newblock {\em Comm. Pure Appl. Math. 76}, 12 (2023), 4075--4086.

\bibitem{chen2009priori}
{\sc Chen, J., Warren, M., and Yuan, Y.}
\newblock A priori estimate for convex solutions to special {L}agrangian
  equations and its application.
\newblock {\em Comm. Pure Appl. Math. 62}, 4 (2009), 583--595.

\bibitem{cirantpayne2021matheng}
{\sc Cirant, M., and Payne, K.~R.}
\newblock Comparison principles for viscosity solutions of elliptic branches of
  fully nonlinear equations independent of the gradient.
\newblock {\em Math. Eng. 3}, 4 (2021), Paper No. 030, 45.

\bibitem{CJY2020Cambridge}
{\sc Collins, T.~C., Jacob, A., and Yau, S.-T.}
\newblock {$(1,1)$} forms with specified {L}agrangian phase: a priori estimates
  and algebraic obstructions.
\newblock {\em Camb. J. Math. 8}, 2 (2020), 407--452.

\bibitem{CollinsPicardWu}
{\sc Collins, T.~C., Picard, S., and Wu, X.}
\newblock Concavity of the {L}agrangian phase operator and applications.
\newblock {\em Calc. Var. Partial Differential Equations 56}, 4 (2017), Paper
  No. 89, 22.

\bibitem{CollinsYau2021AnnPDE}
{\sc Collins, T.~C., and Yau, S.-T.}
\newblock Moment maps, nonlinear {PDE} and stability in mirror symmetry, {I}:
  geodesics.
\newblock {\em Ann. PDE 7}, 1 (2021), Paper No. 11, 73.

\bibitem{DM2023ME}
{\sc Deng, B., and Ma, X.}
\newblock Gradient estimates for the solutions of higher order curvature
  equations with prescribed contact angle.
\newblock {\em Math. Eng. 5}, 6 (2023), Paper No. 093, 13.

\bibitem{DinewAPDE19}
{\sc Dinew, S.~a., Do, H.-S., and T\^{o}, T.~D.}
\newblock A viscosity approach to the {D}irichlet problem for degenerate
  complex {H}essian-type equations.
\newblock {\em Anal. PDE 12}, 2 (2019), 505--535.

\bibitem{FuYauZhang2021}
{\sc Fu, J., Yau, S.-T., and Zhang, D.}
\newblock {A} deformed {H}ermitian {Y}ang-{M}ills {F}low.
\newblock {\em arxiv: 2105.13576, accepted by J. Differential Geom.\/}.

\bibitem{gilbarg2015elliptic}
{\sc Gilbarg, D., and Trudinger, N.~S.}
\newblock {\em Elliptic partial differential equations of second order},
  second~ed., vol.~224 of {\em Grundlehren der mathematischen Wissenschaften
  [Fundamental Principles of Mathematical Sciences]}.
\newblock Springer-Verlag, Berlin, 1983.

\bibitem{GT}
{\sc Gilbarg, D., and Trudinger, N.~S.}
\newblock {\em Elliptic partial differential equations of second order}.
\newblock Classics in Mathematics. Springer-Verlag, Berlin, 2001.
\newblock Reprint of the 1998 edition.

\bibitem{GuanZhang}
{\sc Guan, P., and Zhang, X.}
\newblock A class of curvature type equations.
\newblock {\em Pure Appl. Math. Q. 17}, 3 (2021), 865--907.

\bibitem{HarveyLawson09CPAM}
{\sc Harvey, F.~R., and Lawson, Jr., H.~B.}
\newblock Dirichlet duality and the nonlinear {D}irichlet problem.
\newblock {\em Comm. Pure Appl. Math. 62}, 3 (2009), 396--443.

\bibitem{harveylawson2021cvpde}
{\sc Harvey, F.~R., and Lawson, Jr., H.~B.}
\newblock Pseudoconvexity for the special {L}agrangian potential equation.
\newblock {\em Calc. Var. Partial Differential Equations 60}, 1 (2021), Paper
  No. 6, 37.

\bibitem{harvey1982calibrated}
{\sc Harvey, R., and Lawson, Jr., H.~B.}
\newblock Calibrated geometries.
\newblock {\em Acta Math. 148\/} (1982), 47--157.

\bibitem{JacobYau2017MathAnn}
{\sc Jacob, A., and Yau, S.-T.}
\newblock A special {L}agrangian type equation for holomorphic line bundles.
\newblock {\em Math. Ann. 369}, 1-2 (2017), 869--898.

\bibitem{Krylov}
{\sc Krylov, N.~V.}
\newblock On the general notion of fully nonlinear second-order elliptic
  equations.
\newblock {\em Trans. Amer. Math. Soc. 347}, 3 (1995), 857--895.

\bibitem{LiebermanTrudinger}
{\sc Lieberman, G.~M., and Trudinger, N.~S.}
\newblock Nonlinear oblique boundary value problems for nonlinear elliptic
  equations.
\newblock {\em Trans. Amer. Math. Soc. 295}, 2 (1986), 509--546.

\bibitem{Lin2023Adv}
{\sc Lin, C.-M.}
\newblock The deformed {H}ermitian-{Y}ang-{M}ills equation, the
  {P}ositivstellensatz, and the solvability.
\newblock {\em Adv. Math. 433\/} (2023), Paper No. 109312, 71.

\bibitem{LTU1986}
{\sc Lions, P.-L., Trudinger, N.~S., and Urbas, J. I.~E.}
\newblock The {N}eumann problem for equations of {M}onge-{A}mp\`ere type.
\newblock {\em Comm. Pure Appl. Math. 39}, 4 (1986), 539--563.

\bibitem{LTU1986CPAM}
{\sc Lions, P.-L., Trudinger, N.~S., and Urbas, J. I.~E.}
\newblock The {N}eumann problem for equations of {M}onge-{A}mp\`ere type.
\newblock {\em Comm. Pure Appl. Math. 39}, 4 (1986), 539--563.

\bibitem{lu2023interior}
{\sc Lu, S.}
\newblock Interior $c^{2}$ estimate for hessian quotient equation in dimension
  three.
\newblock {\em arXiv preprint arXiv:2311.05835\/} (2023).

\bibitem{MaQiu2019CMP}
{\sc Ma, X., and Qiu, G.}
\newblock The {N}eumann problem for {H}essian equations.
\newblock {\em Comm. Math. Phys. 366}, 1 (2019), 1--28.

\bibitem{MaXuAdv16}
{\sc Ma, X., and Xu, J.}
\newblock Gradient estimates of mean curvature equations with {N}eumann
  boundary value problems.
\newblock {\em Adv. Math. 290\/} (2016), 1010--1039.

\bibitem{Pingali2022APDE}
{\sc Pingali, V.~P.}
\newblock The deformed {H}ermitian {Y}ang-{M}ills equation on three-folds.
\newblock {\em Anal. PDE 15}, 4 (2022), 921--935.

\bibitem{qiu2019interior}
{\sc Qiu, G.}
\newblock Interior curvature estimates for hypersurfaces of prescribing scalar
  curvature in dimension three.
\newblock {\em arXiv:1901.07791, accepted by Amer. J. Math.\/}.

\bibitem{qiu2017interior}
{\sc Qiu, G.}
\newblock Interior hessian estimates for sigma-2 equations in dimension three.
\newblock {\em arXiv preprint arXiv:1711.00948, accepted by Front. Math.\/}.

\bibitem{QiuXia19}
{\sc Qiu, G., and Xia, C.}
\newblock Classical {N}eumann problems for {H}essian equations and
  {A}lexandrov-{F}enchel's inequalities.
\newblock {\em Int. Math. Res. Not. IMRN}, 20 (2019), 6285--6303.

\bibitem{TrudingerConj}
{\sc Trudinger, N.~S.}
\newblock On degenerate fully nonlinear elliptic equations in balls.
\newblock {\em Bull. Austral. Math. Soc. 35}, 2 (1987), 299--307.

\bibitem{WangHuangBaoCVPDE}
{\sc Wang, C., Huang, R., and Bao, J.}
\newblock On the second boundary value problem for {L}agrangian mean curvature
  equation.
\newblock {\em Calc. Var. Partial Differential Equations 62}, 3 (2023), Paper
  No. 74, 30.

\bibitem{WY11}
{\sc Wang, D., and Yuan, Y.}
\newblock Hessian estimates for special {L}agrangian equations with critical
  and supercritical phases in general dimensions.
\newblock {\em Amer. J. Math. 136}, 2 (2014), 481--499.

\bibitem{Wang2019IJM}
{\sc Wang, J.}
\newblock The {N}eumann problem for complex special {L}agrangian equations with
  critical phase.
\newblock {\em Internat. J. Math. 30}, 9 (2019), 1950043, 28.

\bibitem{Wang2019CMS}
{\sc Wang, J.}
\newblock The {N}eumann problem for special {L}agrangian equations with
  critical phase.
\newblock {\em Commun. Math. Stat. 7}, 3 (2019), 329--361.

\bibitem{Wang2022ANS}
{\sc Wang, P.}
\newblock Gradient estimate of the solutions to {H}essian equations with
  oblique boundary value.
\newblock {\em Adv. Nonlinear Stud. 22}, 1 (2022), 469--483.

\bibitem{warren2009hessian}
{\sc Warren, M., and Yuan, Y.}
\newblock Hessian estimates for the sigma-2 equation in dimension 3.
\newblock {\em Comm. Pure Appl. Math. 62}, 3 (2009), 305--321.

\bibitem{WarrenYuan2010AJM}
{\sc Warren, M., and Yuan, Y.}
\newblock Hessian and gradient estimates for three dimensional special
  {L}agrangian equations with large phase.
\newblock {\em Amer. J. Math. 132}, 3 (2010), 751--770.

\bibitem{Yuan16Slag}
{\sc Yuan, Y.}
\newblock Notes on special lagrangian equations, lecture notes, 2016.

\bibitem{zhou2023hessian}
{\sc Zhou, X.}
\newblock Hessian estimates for lagrangian mean curvature equation with sharp
  lipschitz phase.
\newblock {\em arXiv preprint arXiv:2311.13867\/} (2023).

\bibitem{zhou2023notes}
{\sc Zhou, X.}
\newblock Notes on generalized special lagrangian equation.
\newblock {\em arXiv preprint arXiv:2311.14260\/} (2023).

\end{thebibliography}
\end{document}